\documentclass[11pt,reqno]{amsart}
\usepackage{amsmath}
\usepackage{amsthm}
\usepackage{amssymb}
\usepackage{enumerate}
\usepackage{amscd}
\usepackage{color}
\usepackage{pb-diagram}
\usepackage{graphicx}

\theoremstyle{plain}

\newtheorem{lemma}{Lemma}[section]
\newtheorem{prop}[lemma]{Proposition}

\newtheorem{theorem}{Theorem}

\newtheorem*{mocon}{Min--Oo Conjecture}
\newtheorem*{prop*}{Proposition}
\theoremstyle{definition}
\newtheorem{defn}[lemma]{Definition}
\newtheorem{remark}[lemma]{Remark}
\newtheorem{ex}[lemma]{Example}

\newtheorem*{ack}{Acknowledgement}
\theoremstyle{remark}

\newtheorem*{pfthmMO2}{Proof of Theorem~\ref{th:MO2}}

\newtheorem*{pfthmsmax2}{Proof of Theorem~\ref{th:smax2}}

\newtheorem*{pfleHopfbdy}{Proof of Lemma~\ref{le:Hopfbdy}}

\newtheorem*{pfleHopf2}{Proof of Lemma~\ref{le:Hopf2}}

\newtheorem*{pfthmkcvx}{Proof of Theorem~\ref{th:kcvx}}
\newtheorem*{pfthmkcvxgraph}{Proof of Theorem~\ref{th:kcvxgraph}}
\newtheorem*{pfthmhypmeanPDE}{Proof of Theorem~\ref{th:hypmeanPDE}}
\newtheorem*{pfthmhypmean}{Proof of Theorem~\ref{th:hypmean}}

\numberwithin{equation}{section}
\newenvironment{enumeratei}{\begin{enumerate}[\upshape (i)]}{\end{enumerate}}

\newcommand{\vol}{\textup{Vol}}

\begin{document}
\title[Rigidity theorem on hemispheres]{Rigidity Theorems on Hemispheres in Non--positive Space Forms}
\author{Lan--Hsuan Huang}
\address{Department of Mathematics\\
					Columbia University\\
					New York, NY 10027.}
\email{lhhuang@math.columbia.edu}
\author{Damin Wu}
\address{Department of Mathematics \\
                The Ohio State University \\
                1179 University Drive, Newark, OH 43055.}
\email{dwu@math.ohio-state.edu}

\begin{abstract}
We study the curvature condition which uniquely characterizes the hemisphere. In particular, we prove the Min--Oo conjecture for hypersurfaces in Euclidean space and hyperbolic space. 
\end{abstract}
\maketitle

\section{Introduction} 
There have been many results on characterization of manifolds with non-negative scalar curvature.  
One of the most important theorems is the Positive Mass Theorem proved by Schoen--Yau~\cite{SY79, SY81}, and later by Witten~\cite{Witten}. A special case of their theorem tells us that, if $(M^3,g)$ is asymptotically flat with non-negative scalar curvature, then the ADM mass defined at each end is non-negative; furthermore, if the ADM mass is zero for one end, then $(M^3,g)$ is isometric to Euclidean space. 

For a manifold $M$ with boundary $\partial M$, one can ask a similar question: Under what conditions on $\partial M$ is $M$ isometric to a standard model? This question is, in fact, related to the concept of quasi-local mass in general relativity. Quasi-local mass is a quantity defined on $\partial M$ which measures the energy content of $M$. Shi and Tam \cite{ST} proved that, for a three-manifold $M$ of non-negative scalar curvature, if $\partial M$ has positive Gauss curvature, and if the Brown--York mass is zero, i.e.
\[
		\int_{\partial M} ( H_0 - H ) \, d\sigma = 0, 
\]	
then $M$ is isometric to a domain in Euclidean space. Here $H$ and $H_0$ are the mean curvatures of $\partial M$ induced from $M$ and $\mathbb{R}^3$, respectively.  Miao \cite{M} and Hang--Wang \cite{HW1} also prove some rigidity results on a flat region in Euclidean space under different assumptions.

Besides a flat region in Euclidean space, one can also consider standard spheres as the standard model.
There have been several attempts to understand what properties can uniquely characterize the hemisphere $\mathbb{S}^{n}_+$. The following conjecture was proposed by Min--Oo in 1995.
\begin{mocon}
Le $M$ be an $n$-dimensional compact manifold with boundary $\partial M$. Assume that $M$ has scalar curvature $R\ge n(n-1)$. Furthermore, assume that $\partial M$ is isometric to the unit sphere $\mathbb{S}^{n-1}$, and that $\partial M$ is totally geodesic in $M$. Then $M$ is isometric to the hemisphere $\mathbb{S}^{n}_+$. 
\end{mocon}
While the conjecture is still open, 
some partial results have been obtained. Hang--Wang \cite{HW1, HW2} proved the conjecture under the stronger condition that, either $g$ is conformal to the standard sphere metric, or the Ricci curvature $\mbox{Ric} \ge (n-1)g$. Eichmair \cite{E} proved the conjecture for $n=3$ by assuming $\mbox{Ric} > 0$ in $M$ and an isoperimetric condition on $\partial M$. For other recent results on rigidity, see, for example, \cite{BBEN} and \cite{BBN}.

Let us try to understand the assumptions in the Min--Oo conjecture. Obviously the condition $R \ge n(n-1)$ is necessary, because one can otherwise perturb the hemisphere at an interior point so that $R\ge n(n-1) - \epsilon$, for some small $\epsilon >0$, without changing the assumptions on the boundary. However, it seems that the assumptions on the boundary can possibly be weakened or replaced by other conditions. For example, under the assumption that $\mbox{Ric} \ge (n-1)g$, Hang--Wang~\cite{HW2} relaxed the totally geodesic condition on $\partial M$ to be 
the condition that $\partial M$ is convex in $M$.

In this paper, we study the hypersurfaces in Euclidean space and hyperbolic space, and obtain several curvature conditions which characterize the hemisphere.
We are able to drop the totally geodesic condition on $\partial M$, and also relax the isometry condition on $\partial M$. 
In particular, we prove the Min--Oo conjecture for the hypersurfaces in these non-positive space forms. Our method in fact works for a more general situation than the incorporation condition we state below. Roughly speaking, the proofs work for any compact hypersurface $M$ with boundary $\partial M$, as long as the unit $n$-sphere can travel through $\partial M$.


Let $M$ be a smooth hypersurface in $\mathbb{R}^{n+1}$. Denote by $\kappa_i$, $i=1,\ldots,n$, the principle curvatures of $M$. We define, for each $1 \le k \le n$, the \emph{$k$-th mean curvature} of $M$ to be
\[
   \sigma_k(\kappa) = \sum_{1 \le i_1 < \cdots < i_k \le n} \kappa_{i_1}\cdots \kappa_{i_k}.
\]
In particular, $\sigma_1(\kappa)$, $2 \sigma_2(\kappa)$, and $\sigma_n(\kappa)$ are the mean curvature, the scalar curvature, and the Gauss--Kronecker curvature of $M$, respectively. Our convention of the mean curvature is that the unit $n$-sphere has mean curvature $n$ with respect to the inward unit normal vector. We say that $M$ is \emph{$q$-convex}, $1 \le q \le n$, if its $j$-th mean curvature is positive, for all $j = 1, \ldots, q$. These definitions also make sense for a $C^2$ hypersurface. 

Let $B_1$ be the open unit ball in the hyperplane $\mathbb{R}^n \times \{0\}$ centered at the origin. Denote by $\mathcal{C}_+$ the upper solid hemicylinder, i.e.,
\[
   \mathcal{C}_+ 
		= \{ (x^1, \dots, x^{n+1} ) \in \mathbb{R}^{n+1}: r \le 1 \mbox{ and } x^{n+1} > 0 \}.
\]
Throughout this paper, we denote 
\[
    r = \sqrt{(x^1)^2 + \cdots + (x^n)^2}, \qquad \textup{for all $(x^1,\ldots,x^n) \in \mathbb{R}^n$}.
\]
It is convenient to introduce the following definition.
\begin{defn} \label{de:incp}
Let $M \subset \mathbb{R}^{n+1}$ be a compact hypersurface which is $C^2$ up to the boundary $\partial M$. We say that $M$ satisfies the \emph{incorporation} condition, if $M$ satisfies the following three conditions
\begin{enumeratei}
  \item  \label{de:incp:1} $\partial M$ is diffeomorphic to $\mathbb{S}^{n-1}$.
  \item  \label{de:incp:2} $\partial M \subset \mathbb{R}^n \times \{0\}$, and $B_1$ is contained in the region enclosed by $\partial M$ in $\mathbb{R}^n \times \{0\}$.
  \item \label{de:incp:3} $M \cap \mathcal{C}_+ = \emptyset$.
\end{enumeratei}
\end{defn}
The rigidity theorem in Euclidean space is as follows:
\begin{theorem} \label{th:kcvx}
Let $M \subset \mathbb{R}^{n+1}$ be a compact $C^2$ hypersurface with boundary $\partial M$. Assume that $M$ satisfies the incorporation condition. Suppose for some integer $1 \le k \le n$ that
\begin{equation} \label{eq:kvx}
		\sigma_k (\kappa) \ge \binom{n}{k},
\end{equation}
and that $M$ is $k$-convex if $k \ge 3$. Then $M$ is isometric to the hemisphere $\mathbb{S}^n_+$.
\end{theorem}
Note that for $k=2$, \eqref{eq:kvx} is equivalent to that $R \ge n(n-1)$. Also the convexity is not required for $k=2$. This theorem in particular settles the Min--Oo conjecture for the hypersurfaces in Euclidean space. Furthermore, if $M$ is a graph of a function, then the incorporation condition can be dropped. Here is the result for the scalar curvature.
\begin{theorem} \label{th:MO2}
		Let $u \in C^2 (B_1)\cap C^0(\bar{B}_1)$, and $M_u$ be the graph of $u$ over $\bar{B}_1$ in $\mathbb{R}^{n+1}$. If $M_u$ has induced scalar curvature $R \ge n(n-1)$, then $M_u$ is isometric to the hemisphere $\mathbb{S}_+^n$. 
\end{theorem}
Note that the function $u$ here can be prescribed any continuous boundary value. 
Theorem~\ref{th:MO2} is generalized to the statement for the $k$-th mean curvature (see Theorem~\ref{th:kcvxgraph}, in Section~\ref{se:Eu}, for details.).

Next, we consider the upper half-space model of the hyperbolic space $\mathbb{H}^{n+1}$. We consider the hypersurfaces satisfying the \emph{hyperbolic incorporation condition} (see Definition~\ref{de:hypincp} for details.). This definition is the same as Definition~\ref{de:incp}, except that $\mathbb{R}^{n+1}$ is replaced by $\mathbb{H}^{n+1}$, that the hyperplane $\mathbb{R}^n \times \{0\}$ in Definition~\ref{de:incp} \eqref{de:incp:2} is replaced by $\mathbb{R}^n \times \{1\}$, and that the cylinder $\mathcal{C}_+$ in \eqref{de:incp:3} is replaced by the following upper solid hemicone.
\[
		\mathfrak{C}_+= \{ ( x^1, \dots, x^{n+1} ) \in \mathbb{H}^{n+1} : x^{n+1} \ge r \mbox{ and } x^{n+1} > 1\}.
\]
We obtain similar rigidity results in $\mathbb{H}^{n+1}$.
\begin{theorem} \label{th:hypkcvx}
  Let $M \subset \mathbb{H}^{n+1}$ be a compact $C^2$ hypersurface with boundary $\partial M$. Assume that $M$ satisfies the hyperbolic incorporation condition. Suppose for some integer $1 \le k \le n$ that 
 \begin{equation} \label{eq:hypkcvx}
     \sigma_k(\kappa) \ge 2^{k/2}\binom{n}{k},
  \end{equation}
  and that $M$ is $k$-convex if $k \ge 3$. Then $M$ is isometric to $\mathbb{S}^n_+$.
\end{theorem}
Let us remark that when $k=2$, the condition \eqref{eq:hypkcvx} is equivalent to $R \ge n(n-1)$, in view of the Gauss equation \eqref{eq:hypR2}. The convexity is not required for this case. This proves the Min--Oo conjecture in  hyperbolic space. In particular, when $M$ is graphical, we can relax the hyperbolic incorporation condition. The rigidity result for the scalar curvature is as follows:
\begin{theorem} \label{th:hyp2graph}
  Let $u \in C^2(B_1)\cap C^1(\bar{B}_1)$ satisfy that $u > 0$ in $B_1$ and $u = 1$ on $\partial B_1$. If the graph of $u$, denoted by $M_u$, has hyperbolic scalar curvature $R \ge n(n-1)$, then $M_u$ is isometric to $\mathbb{S}_+^n$.  
\end{theorem}
Here the assumption on $u$ regarding boundary regularity is stronger than that in Theorem~\ref{th:MO2}. This is due to the difference of the geometry, which we will briefly mention below. As a remark, a more general statement of Theorem~\ref{th:hyp2graph} for the $k$-th mean curvature is proved in Section~\ref{se:Hy} (Theorem~\ref{th:hypkgraph}).

An important observation in this paper is that, we can reduce the problem of scalar curvature (or of $\sigma_k$ in general) to that of mean curvature, either by the Gauss equation, or by the Newton--Maclaurin inequalities (see \cite{HLP}, for example). The advantage is that the mean curvature operator is relatively easy to handle, especially in Euclidean space. 

The approach is relatively difficult in hyperbolic space. To begin with, one must figure out a \emph{model} geodesic sphere which plays the same role as the unit $n$-sphere in Euclidean space. Besides, the hyperbolic mean curvature operator is significantly different from that in Euclidean space. For example, the standard comparison principle (see, for example, Gilbarg--Trudinger~\cite[p. 263--267]{GT}) does \emph{not} apply to the hyperbolic mean curvature operator. We in fact construct a counter example in Example~\ref{ex:hypH}. 

Our main tool to derive the hyperbolic rigidity theorem is the strong maximum principles including the boundary point lemma. These principles are known for general quasilinear operators (see, for example, Serrin~\cite{Serrin} and Pucci--Serrin~\cite{PS}). But as indicated in Example~\ref{ex:hypH}, in the actual applications one has to be careful for different requirements on the ellipticity, regularity, and coefficients. For completeness, we state and prove the strong maximum principles in a form we need.

The rest of the paper is arranged as follows. In Section~\ref{se:Eu}, we prove the rigidity theorems for the hypersurfaces in Euclidean space. In Section~\ref{se:Hy}, we study the geometry of hypersurfaces in hyperbolic space, and prove the hyperbolic version of rigidity theorems.
Finally, we include in the Appendix the complete proof of the strong maximum principles for the mean curvature operators.

\begin{ack}
  Both the authors would like to thank Professor Rick Schoen for many helpful discussions. The second named author would also like to thank the warm hospitality of Stanford University, and the support of The Ohio State University at Newark. We would like to thank the referee for bringing our attention to some other recent results on rigidity such as \cite{BBEN} and \cite{BBN}.
\end{ack}

\section{Hypersurfaces in Euclidean Space} \label{se:Eu}

Let $M$ be a $C^2$ hypersurface in $\mathbb{R}^{n+1}$, and $A = (A_i^j)$ be the shape operator of $M$ with the eigenvalues $\kappa_i$ for $ 1 \le i \le n$ . We define the \emph{$k$-th mean curvature of $M$}, denoted by $\sigma_k(A)$ or $\sigma_k(\kappa)$, to be the $k$-th symmetric polynomial on $\kappa = (\kappa_1,\ldots,\kappa_n)$, i.e.,
\[
    \sigma_k(A) = \sigma_k(\kappa) = \sum_{1 \le i_1 < \cdots < i_k \le n} \kappa_{i_1}\cdots\kappa_{i_k}.
\]
In particular, if $M$ is smooth, $\sigma_1(\kappa)$, $2 \sigma_2(\kappa)$, and $\sigma_n(\kappa)$ are the mean curvature, the scalar curvature, and the Gauss--Kronecker curvature of $M$, respectively. Therefore, we can similarly define for a $C^2$ hypersurface, its mean curvature, scalar curvature, and Gauss--Kronecker curvature to be $\sigma_1(\kappa)$, $2 \sigma_2(\kappa)$, and $\sigma_n(\kappa)$, respectively. We say a $C^2$ hypersurface is \emph{$l$-convex}, $1 \le l \le n$, if its $j$-th mean curvature is positive for all $j = 1, \ldots, l$. 

In this section, we may interchangeably use $\sigma_1(A)$, $\sigma_1(\kappa)$, and $H_0$ to denote the mean curvature of a hypersurface. The notation $B_a$ stands for the open ball in $\mathbb{R}^n \times \{0\}$ of radius $a>0$ centered at the origin. Unless otherwise indicated, we always denote 
\[
    r = \sqrt{(x^1)^2 + \cdots + (x^n)^2}.
\]


Let us first proceed to prove Theorem~\ref{th:MO2}, in which the hypersurface is graphical. 
Let $u \in C^2(B_1)\cap C^0(\bar{B}_1)$ and $M_u$ be the graph of $u$ over $\bar{B}_1$ in $\mathbb{R}^{n+1}$, i.e.,
\[
		M_u = \{(x^1,\ldots, x^n, x^{n+1}) : x^{n+1} = u ( x^1, \dots, x^n ), \mbox{  for all } r \le 1 \}.
\]
 The mean curvature of $M_u$, with respect to the upward unit normal vector, is 
\[
		\sigma_1(\kappa) =  H_0 (u) = \sum_{i=1}^n \frac{\partial}{\partial x^i} \left(\frac{\partial u/\partial x^i}{\sqrt{1 + |Du|^2}} \right).
\]
Below a simple estimate of the total mean curvature is derived.
\begin{prop} \label{pr:meangraph}
  \[
     \left|\int_{B_1} H_0(u) dx \right| \le n \vol(B_1).
  \]
\end{prop}
\begin{proof}
It follows from the divergence theorem that, for any $0 < a < 1$,
\begin{align*}
   \left|\int_{B_a} H_0 (u) dx \right|
   & = \left|\int_{\partial B_a} \frac{D u \cdot (y/a)}{\sqrt{1 + |Du|^2}} d y \right|\\
   & \le \int_{\partial B_a} \frac{|D u|}{\sqrt{1 + |Du|^2}} dy \\
   & \le \vol( \partial B_a) = \frac{n}{a} \vol( B_a) .
\end{align*}
Letting $a$ tend to $1$ yields the result.
\end{proof}

\begin{pfthmMO2}
Let us invoke a useful identity:
\begin{equation} \label{eq:LanNewton}
   \left(\frac{H_0}{n}\right)^2 = \frac{\sigma_2(A)}{\binom{n}{2}}  + \frac{ |\mathring{A} |^2}{n(n-1)}.
\end{equation}
We denote by $\mathring{A}$ the trace-free part of $A$, i.e.,
\begin{equation} \label{eq:tf}
	( \mathring{A} )^j_i = A^j_i - \frac{H_0}{n} \delta^j_i.
\end{equation}
The identity \eqref{eq:LanNewton} follows immediately from substituting \eqref{eq:tf} into
\[
   H_0^2 - 2 \sigma_2(A) = |A|^2.
\]

Now if $2\sigma_2 ( A ) \ge n(n-1)$ everywhere, we have by \eqref{eq:LanNewton} that
\[
		H_0^2 \ge n^2 + \frac{n}{n-1}| \mathring{A} |^2 \ge n^2.
\]
Therefore, by continuity, we have either $H_0 \ge n$ everywhere, or $H_0 \le -n$ everywhere.
But in view of Proposition~\ref{pr:meangraph}, we obtain  in fact an identity in either case. Applying \eqref{eq:LanNewton} again yields that $\mathring{A} \equiv 0$, i.e., $M_u$ is totally umbilic with all principal curvatures identically equaling $1$. Therefore, $M_u$ is isometric to the hemisphere $\mathbb{S}^n_+$. 
\qed
\end{pfthmMO2}

Next, we would like to prove Theorem~\ref{th:kcvx}. Recall that $\mathcal{C}_+$ is the upper solid hemicylinder given by
\[
   \mathcal{C}_+ 
		= \{ (x^1, \dots, x^{n+1} ) \in \mathbb{R}^{n+1}: r \le 1 \mbox{ and } x^{n+1} > 0 \}.
\]
Let $M$ be a $C^2$ hypersurface with boundary $\partial M$. We say $M$ satisfies the incorporation condition, if $M$ has the following three conditions
\begin{enumeratei}
  \item   $\partial M$ is diffeomorphic to $\mathbb{S}^{n-1}$.
  \item  $\partial M \subset \mathbb{R}^n \times \{0\}$, and $B_1$ is contained in the region enclosed by $\partial M$ in $\mathbb{R}^n \times \{0\}$.
  \item  $M \cap \mathcal{C}_+ = \emptyset$.
\end{enumeratei}


The following lemma settles the mean curvature case in Theorem~\ref{th:kcvx}. 
 
\begin{lemma} \label{le:mean}
Let $M \subset \mathbb{R}^{n+1}$ be a compact hypersurface with boundary $\partial M$. Assume that $M$ satisfies the incorporation condition. If the mean curvature of $M$ satisfies that $|H_0| \ge n$ everywhere, then $M$ is isometric to the hemisphere $\mathbb{S}^n_+$. 
\end{lemma}
\begin{proof}
Without loss of generality, we can assume that $H_0\ge n$ with respect to a non-vanishing unit normal vector field $\nu$ on $M$. More precisely, $\nu$ is the inward unit normal vector if we ``close up'' $M$ by adding the flat region enclosed by $\partial M$ in $\mathbb{R}^n \times \{ 0 \}$. Let $S(q)$ be the unit $n$-sphere in $\mathbb{R}^{n+1}$ centered at $(0,\ldots, 0,q)$, for each $q \in \mathbb{R}$. Since $M$ is compact, we can start with a very large $q$ so that $S(q)$ has no intersection with $M$. Then, we continuously decreases $q$ until $S(q)$ begins to intersect $M$. We denote $q = q_0$ for this moment.

We \emph{assert} that if $S(q_0)$ is tangent to $M$ at an interior point, then $M$ must be a portion of $S(q_0)$, and hence, $\partial M = \partial B_1$ by the incorporation condition; therefore $M$ is isometric to the hemisphere. 
Indeed, let $V = M \cap S(q_0)$. Obviously $V$ is a nonempty closed subset in $M$. If $V \neq M$, then there exists an interior point $p$ of $M$ such that $p \in \partial V$. 
Locally near $p$, both $S(q_0)$ and $M$ can be written as graphs over $T_p M$. 
Note that the mean curvature of the graph of $S(q_0)$ is equal to $n$ (instead of $-n$) with respect to the unit normal $\nu$. This is guaranteed by the incorporation condition \eqref{de:incp:2} and \eqref{de:incp:3}. Applying part~\eqref{th:smax2:1} of Theorem~\ref{th:smax2} yields that $M$ must coincide with $S(q_0)$ over the small neighborhood of $p$. This contradicts the assumption $p \in \partial V$. Hence, $V = M$ and the assertion is proved. 

The assertion will imply that $\partial M \cap \partial B_1 \neq \emptyset$, and $q_0 = 0$. Suppose that $\partial M \cap \partial B_1 = \emptyset$. Then by the incorporation condition $S(q_0)$ has to tangent to $M$ at some interior point. By the assertion $M$ is a portion of $S(q_0)$ and $\partial M = \partial B_1$. It is a contradiction. Thus, we have $q_0 \ge 0$. If $q_0 > 0$, then again $S(q_0)$ must intersect $M$ at the interior, by the incorporation \eqref{de:incp:3}. We arrive a contradiction again by the assertion.

It remains to rule out the case that $S(0)$ intersects $M$ only at $\partial M$. For any $x_0 \in \partial M \cap \partial B_1$, we can locally write $M$ and $S(0)$ as graphs $\psi$ and $\varphi$ (with $\varphi \ge 0$) over $T_{x_0} \mathit{C}$, respectively. Here $\mathit{C} $ is the cylinder
\begin{equation} \label{eq:defCy}
   \mathit{C} = \{ (x^1, \ldots, x^{n+1}) \in \mathbb{R}^{n+1} \mid r = 1 \}.
\end{equation}
We have by the second part of Theorem~\ref{th:smax2} that
\[
    \frac{\partial (\psi - \varphi)}{\partial x^{n+1}}(x_0) > 0.
\]
Let us also write $\mathit{C}$ locally as the graph $\bar{\varphi}$ over $T_{x_0} \mathit{C}$. Because $S(0)$ is tangent to $\mathit{C}$ at $x_0$, $\bar{\varphi}$ and  $\varphi$ have the same derivatives in the direction of $\partial/\partial x^{n+1}$, and therefore,
\[
    \frac{\partial (\psi - \bar{\varphi})}{\partial x^{n+1}}(x_0) = \frac{\partial (\psi - \varphi)}{\partial x^{n+1}}(x_0) > 0. 
\]
This implies that $\psi(x) < \bar{\varphi}(x)$ for any $x = x_0 - (0,\ldots,0, t)$ with $t > 0$ small. This holds for any $x_0 \in \partial M \cap \partial B_1$. Hence, there exists a small constant $\delta >0$ such that
\[
   M \cap \mathcal{C}_{-\delta} = \emptyset,
\]
in which
\[
		\mathcal{C}_{-\delta} = \{ (x^1, \dots, x^{n+1} ) \in \mathbb{R}^{n+1}: r \le 1 \mbox{ and } - \delta < x^{n+1} < 0 \}.
\]
Then, there exists a small constant $\epsilon > 0$ such that $S(q)$ has no intersection with $M$ for any $- \epsilon < q < 0$. Thus, we can continuously decrease $q$ until $S(q)$ is tangent to some interior point of $M$. We get a contradiction by the previous assertion. This completes the proof.
\end{proof}

\begin{pfthmkcvx}
It remains to show the theorem for $k \ge 2$. If $k =2$, by \eqref{eq:LanNewton} we have
\[
		| H_0 | \ge n.
\]
The result then follows immediately from Lemma~\ref{le:mean}. 
Assume for some $k \ge 3$ that $M$ is $k$-convex and $\sigma_k (\kappa) \ge \binom{n}{k}$. Recall Maclaurin's inequality states that, 
\begin{equation} \label{eq:Mac}
   \left[ \frac{ \sigma_k (\lambda)  }{ \binom{n}{k} } \right]^{1/k}\le \left[ \frac{ \sigma_{k-1} (\lambda)  }{ \binom{n}{k -1} } \right]^{1/(k -1)} \le \cdots \le \frac{ \sigma_1(\lambda) }{n},
\end{equation}
for any $\lambda = (\lambda_1, \cdots, \lambda_n) \in \mathbb{R}^n$ with $\sigma_j (\lambda) > 0$ for all $j = 1, \ldots, k$. 
It follows that 
\[
     H_0 = \sigma_1 ( \kappa) \ge n \left[ \frac{\sigma_k(\kappa)}{\binom{n}{k}} \right]^{1/k} \ge n 
\]
everywhere on $M$. Thus, the result is implied by Lemma~\ref{le:mean}.
\qed
\end{pfthmkcvx}

If $M$ is graphical, then the incorporation condition in Theorem~\ref{th:kcvx} can be dropped. 
\begin{theorem} \label{th:kcvxgraph}
  Let $u \in C^2(B_1)\cap C^0(\bar{B}_1)$, and $M_u$ be the graph of $u$ over $\bar{B}_1$ in $\mathbb{R}^{n+1}$. 
  Assume that $M_u$ is $k$-convex for some integer $3 \le k \le n$, and
  \[
     \sigma_k \ge \binom{n}{k}.
  \]
   Then $M_u$ is isometric to $\mathbb{S}^n_+$.
\end{theorem}

\begin{pfthmkcvxgraph}
For $3 \le k \le n$, Maclaurin's inequality implies that
\[
   2 \sigma_2(\kappa) \ge n(n-1) \left[\frac{\sigma_k(\kappa)}{\binom{n}{k}} \right]^{1/k} \ge n(n-1).
\]
The result then follows from Theorem~\ref{th:MO2}.
\qed
\end{pfthmkcvxgraph}

\section{Hypersurfaces in Hyperbolic Space} \label{se:Hy}

We consider the upper half--space model for the hyperbolic space $\mathbb{H}^{n+1}$ with the metric $x_{n+1}^{-2} \delta_{ij}$. Let $M$ be a $C^2$ hypersurface in $\mathbb{H}^{n+1}$ and $A = ( A^j_i )$ be the hyperbolic shape operator. Similar to Euclidean space, we define the \emph{hyperbolic $k$-th mean curvature} of $M$ to be the $k$-th symmetric functions on $A$. As before, $M$ is called \emph{$k$-convex}, if $\sigma_j(A)>0$ for all $1 \le j \le k$. Note that $\sigma_1(A)$ is equal to the hyperbolic mean curvature $H$. 
In contrast to the Euclidean case, the scalar curvature of $M$ induced from $\mathbb{H}^{n+1}$ is defined to be  
\begin{equation} \label{eq:hypR2}
   R = 2 \sigma_2(A) - n (n-1).
\end{equation}
(In the following, we call the induced scalar curvature of a hypersurface in $\mathbb{H}^{n+1}$ the \emph{hyperbolic scalar curvature}.)

The reason is due to the Gauss equation. More precisely, assume that $M$ is a smooth (or at least $C^3$) hypersurface in $\mathbb{H}^{n+1}$. 
For any $p \in M$, we choose an orthonormal basis $\{e_i\}_{i=1}^n$ of $T_p M$. Denote by $\bar{R}_{ijkl}$ and $R_{ijkl}$, respectively, the Riemann curvature tensors of $\mathbb{H}^{n+1}$ and $M$ at $p$ with respect to $\{ e_i\}_{i=1}^n$. Then, by the Gauss equation,
\begin{align} \label{eq:gauss}
		\bar{K}_{ij} = \bar{R}_{ijji} = R_{ijji} - A_{ii}A_{jj} + A_{ij}A_{ij},
\end{align}
where $\bar{K}_{ij}$ is the sectional curvature of $\mathbb{H}^{n+1}$ at $p$. 
Recall that
\[
   \bar{K}_{ij} = - 1 + \delta_{ij}, \qquad \textup{for all $1 \le i, j \le n$}.
\]
Summing \eqref{eq:gauss} over all $i$, $j = 1, \ldots, n$ 
yields that
\begin{equation} \label{eq:GaussR}
		- n(n-1) = R - H^2 + | A|^2 
\end{equation}
where $R$ and $H$ are, respectively, the hyperbolic scalar curvature and mean curvature of $M$. 
On the other hand, we know that
\[
   2 \sigma_2(A) = H^2 - |A|^2. 
\]
This combining \eqref{eq:GaussR} yields that
\[
   R = 2 \sigma_2(A) - n(n-1).
\]

The following simple proposition relates the hyperbolic scalar curvature to the hyperbolic mean curvature.
\begin{prop} \label{pr:hypRH}
   Let $M$ be a $C^2$ hypersurface in $\mathbb{H}^{n+1}$. Denote by $\mathring{A}$ the trace-free part of $A$. Then,
   \begin{equation} \label{eq:hypRH}
      \left(\frac{H}{n}\right)^2 = \frac{| \mathring{A}|^2}{n(n-1)} + \frac{R}{n(n-1)} + 1.
   \end{equation}
   As a consequence, if $R \ge n(n-1)$, then
   \[
        H \ge \sqrt{2}n,
   \]
   where the equality holds if and only if $M$ is totally umbilic with all the principal curvatures identically equaling $\sqrt{2}$.
\end{prop}
\begin{proof}
   The identity \eqref{eq:hypRH} follows immediately from substituting 
   \[
       |A|^2 = | \mathring{A} |^2 + \frac{H^2}{n}
   \]
   into
   \[
       R + n(n-1) = 2 \sigma_2(A) = H^2 - |A|^2. 
   \]
\end{proof}

Let us now look at the graph case. Let $u = u(x^1, \dots, x^n) \in C^2 (B_1)\cap C^0(\bar{B}_1)$. Throughout this section, we denote for $\delta > 0$,
\[
		B_{\delta} = \{ (x^1, \dots, x^n, 0) \in \mathbb{R}^{n+1} \mid r < \delta\},
\]
and
\[
    r = |x| = \sqrt{(x^1)^2 + \cdots (x^n)^2}.
\]
Let $M_u$ be the graph of $u$ over $\bar{B}_1$, i.e.,
\[
		M_u = \{(x^1,\ldots, x^n, x^{n+1}) \mid x^{n+1} = u ( x^1, \dots, x^n ), \mbox{  for all } r \le 1 \}.
\]
Notice that if $u=1$ on $\partial B_1$, then $\partial M_u$ with the induced metric from $\mathbb{H}^{n+1}$ is isometric to the unit $(n-1)$--sphere $\mathbb{S}^{n-1}$. The mean curvature of $M_u$, with respect to the upward unit normal vector, is 
\begin{equation} \label{eq:hypH}
		H(u) = \frac{ n }{ \sqrt{ 1 + | Du |^2 } } + u \sum_{i=1}^n \frac{ \partial } { \partial x^i } \left( \frac{ \partial u/\partial x^i }{ \sqrt{ 1 + | D u |^2 } }  \right).
\end{equation}


A geodesic sphere in $\mathbb{H}^{n+1}$ is an Euclidean sphere which is contained in $\mathbb{H}^{n+1}$. The hyperbolic mean curvature of a geodesic sphere is given by
\[
		H = \frac{q}{a} n, 
\]
where $q$ is the height of the center and $a$ is the radius. Among all the hyperbolic geodesic spheres which pass through $\partial B_1 \times \{1\}$, the one of radius $\sqrt{2}$ centered at $(0,\ldots,0,2)$ has the maximum mean curvature, which is equal to $\sqrt{2}n$. We call this geodesic sphere the \emph{model sphere}. Let 
\[
	v = 2 - \sqrt{ 2 - r^2 }, \qquad \textup{for all $r \le 1$},
\]
and $M_v$ be the graph of $v$. Then, $M_v$ is the portion of the model sphere. By abuse of language, we also refer $M_v$ (or $v$) as the model sphere.

\begin{prop} \label{pr:hypmodel}
  The model sphere $M_v$, endowed with the induced metric, is isometric to the hemisphere $\mathbb{S}^n_+$.
\end{prop}
\begin{proof}
  First, notice that $M_v$ is totally umbilic with all the principle curvatures identically equaling $\sqrt{2}$. 
  Let $\{e_i\}_{i=1}^n$ be an orthonormal basis under which $A$ is diagonalized. Since $M_v$ is smooth, we can apply the Gauss equation \eqref{eq:gauss} to obtain that
  \[
		-1 = \bar{K}_{ij} = K_{ij} - 2, \qquad \textup{for all $1 \le i , j \le n$ and $i \ne j$},
	\]
	where $K_{ij}$ is the sectional curvature of $M_v$. Thus, $M_v$ has constant sectional curvature $1$. Moreover, $M_v$ is simply connected, and $\partial M_v = \partial B_1 \times \{1\}$ which is isometric to $\mathbb{S}^{n-1}$. Therefore, we conclude that $M_v$ is isometric to $\mathbb{S}^n_+$.
\end{proof}

Before proving the theorems, let us remark that the standard comparison principle (see, for example, Gilbarg--Trudinger~\cite[p. 263--267]{GT}) does \emph{not} apply to the hyperbolic mean curvature operator $H(u)$. One reason is that the second order coefficients in $H$ depend on $u$. See below for a counter example.
\begin{ex} \label{ex:hypH}
  Let us compare the model sphere $v$ with the following two functions:
\[
   u_1 \equiv 1 \quad \textup{and \quad $u_2 = 1 + \varepsilon - \sqrt{1 + \varepsilon^2 - r^2}$, \quad for all $r \le 1$},
\]  
where $0 < \varepsilon < 1/2$. Observe that $u_1, u_2 \in C^{\infty}(\bar{B}_1)$ with
\[
   u_1\big|_{\partial B_1} = u_2 \big|_{\partial B_1} = 1 = v\big|_{\partial B_1},
\]
and that 
\[
   H (u_1) = n < \sqrt{2}n = H(v), \quad \textup{and  \quad $H(u_2) = n \frac{(1+\varepsilon)}{\sqrt{1 + \varepsilon^2}} < \sqrt{2}n$.}
\]
Nevertheless, we have
\[
   u_1 \ge v \ge u_2 \qquad \textup{on $\bar{B}_1$}.
\]
\end{ex}



Now we would like to show the rigidity theorem for the mean curvature.
\begin{theorem} \label{th:hypmeanPDE}
  Let $u \in C^2(B_1)\cap C^1(\bar{B}_1)$ such that $u = 1$ on $\partial B_1$. 
  If $H(u) \ge \sqrt{2}n$ where $H$ is given by \eqref{eq:hypH}, then
  \[
      u \equiv v = 2 - \sqrt{2 - r^2}, \qquad \textup{on $\bar{B}_1$}.
  \]
\end{theorem}
\begin{pfthmhypmeanPDE}
  Let us first point out that 
  \[
       u > 0 \qquad \textup{on $\bar{B}_1$},
  \]
because, if $u$ achieves a non-positive minimum at $y_0\in B_1$, then a direct computation shows that $H(u)(y_0) \le n$.

  Let
  \[
     v_q (r) = q - \sqrt{\frac{q^2}{2} - r^2}, \qquad \textup{for all $r = |x| \le \min \left\{1, q/\sqrt{2} \right\}$ },
  \]
  where $q > 0$ is a constant. 
  Then $\{ v_q \}$ defines a family of geodesic spheres, of radius $q/\sqrt{2}$ centered at $(0,\ldots,0,q)$,
   whose hyperbolic mean curvatures are all equal to $\sqrt{2} n$. 
   
   The following two inequalities will be used: First, for all $q \ge \sqrt{2}$ and $q \ne 2$,
  \begin{equation} \label{eq:uvbdy}
     v_{q}\big|_{\partial B_1} = v_{q}(1) = q - \sqrt{\frac{q^2}{2} - 1} > 1 = u\big|_{\partial B_1}.
  \end{equation}
  Second,
  \begin{equation} \label{eq:vcone}
     v_q (r) \ge r, \qquad \textup{for all $q > 0$ and all $r \le q/\sqrt{2}$}. 
  \end{equation}
  
  We start with a very large $q$ so that there is no intersection of the graphs of $v_q$ and $u$. 
  Then, we continuously decrease $q$ until ultimately the graph of $v_q$ begins to intersect the graph of $u$. Denote by $q = q_0$ for this moment. Observe that $q_0 \ge 2$, since $v_{q}$, with $q =2$, is exactly the model sphere $v$ which intersects the graph of $u$ at least on $\partial B_1$. 
  
  We assert that $q_0 = 2$. Suppose that $q_0 > 2$. Then $v_{q_0} \in C^{\infty}(\bar{B}_1)$, and  $v_{q_0} \ge u$ on $\bar{B}_1$ by the above construction.
  In view of \eqref{eq:uvbdy}, $v_{q_0}$ must be equal to $u$ at some interior point in $B_1$. Applying Theorem~\ref{th:smax2} \eqref{th:smax2:1} with $Q = H$ and $V = B_1$ yields that 
  \[
      v_{q_0} \equiv u \qquad \textup{on $\bar{B}_1$}.
  \]
  This contradicts with \eqref{eq:uvbdy}. The assertion is proved.
  
  Thus, we have
  \[
      v_{q_0} = v = 2 - \sqrt{2 - r^2}, \qquad \textup{for all $r \le 1$}.
  \]
  Clearly, $v \in C^{\infty}(\bar{B}_1)$, and $v \ge u$ on $\bar{B}_1$.  
  Note that if $v = u$ at some interior point of $B_1$. Again by Theorem~\ref{th:smax2} we have $u \equiv v$ on $\bar{B}_1$. Therefore, to prove this theorem, it suffices to rule out the remaining case, i.e., 
  \begin{equation} \label{eq:case2}
     v(x) > u(x) \qquad \textup{for all $x$ in $B_1$}.
  \end{equation}

Assume that \eqref{eq:case2} holds. We \textbf{claim} that
there exists a sufficiently small constant $\epsilon > 0$ such that for any $2 > q > 2 - \epsilon$, 
\[
    v_q = q - \sqrt{\frac{q^2}{2} - r^2} > u \qquad \textup{on $\bar{B}_1$}.
\]

Deferring its proof, we first proceed to exclude \eqref{eq:case2}: By the claim we can find a $q \in (2 - \epsilon, 2)$ so that the graph of $v_q$ lies completely above the graph of $u$ over $\bar{B}_1$. Then, we can move the graph of $v_q$ downward, by continuously decreasing $q$. Note that eventually the graph of $v_q$ must intersect the graph of $u$, since $u(0)> 0$ and
\[
   v_q(0) = \min_{B_1} v_q =  \left( 1 - \frac{1}{\sqrt{2}}\right) q \to 0, \qquad \textup{as $q \to 0^+$}.
\]
We denote by $q_1$ for the moment that the graph of $v_{q_1}$ begins to intersect the graph of $u$. There are two cases: 

Case $1^{\circ}$: Suppose that $2 > q_1 \ge \sqrt{2}$. Then $v_{q_1} \in C^{\infty}(B_1) \cap C^0(\bar{B}_1)$ and $v_{q_1} \ge u$ on $\bar{B}_1$. 
Because of \eqref{eq:uvbdy}, $v_{q_1}$ has to equal $u$ at some interior point of $B_1$. By Theorem~\ref{th:smax2} \eqref{th:smax2:1}, we obtain that $v_{q_1} \equiv u$ on $\bar{B}_1$, which is a contradiction.

Case $2^{\circ}$: Suppose that $\sqrt{2} > q_1 > 0$. Then $v_{q_1} \in C^{\infty}(B_{\tau}) \cap C^0(\bar{B}_{\tau})$ and $v_{q_1} \ge u$ on $\bar{B}_{\tau}$, where $\tau = q_1/ \sqrt{2} < 1$. Note that
\[
  \frac{\partial v_{q_1}}{\partial r}(r) = \frac{r}{\sqrt{q_1^2/2 - r^2}} \to + \infty, \qquad \textup{as $r \to \tau$}.
\]
This implies that the graph of $v_{q_1}$ cannot touch the graph of $u$ over a point on $\partial B_{\tau}$, for otherwise $u$ would not be in $C^2(B_1)$. Thus, $v_{q_1} = u$ at some interior point of $B_{\tau}$. However, applying Theorem~\ref{th:smax2} \eqref{th:smax2:1} yields that $v_{q_1} \equiv u$ on $\bar{B}_{\tau}$; this again violates the assumption that $u \in C^2(B_1)$. Therefore, combining the two cases, we have ruled out \eqref{eq:case2}, with the aid of the claim.

Let us now settle the claim. A key observation is 
that there is a small constant $\delta > 0$ such that 
\begin{equation}\label{eq:cone}
   u(x) < |x| = r, \qquad \textup{for all $1 - \delta \le |x| < 1$}.
\end{equation}
In fact, by \eqref{eq:case2} and applying Theorem~\ref{th:smax2} \eqref{th:smax2:2} yields that, 
\[
   0 < (Du - Dv)(x_0) \cdot x_0 = (Du - Dr)(x_0)\cdot x_0, \qquad \textup{for each $x_0 \in \partial B_1$},
\]
where we use the fact that
\[
    Dv(x_0) = Dr(x_0) \qquad \textup{for all $x_0 \in \partial B_1$}.
\]
Then by continuity, there exists a sufficiently small $\delta > 0$ such that
\[
   D (u-r)(x) \cdot \frac{x}{|x|} > 0, \qquad \textup{for all $1 - \delta \le |x| \le 1$}.
\]
Since $u \in C^1(\bar{B}_1)$, we have by the mean value theorem that
\[
   u (x) - r(x) < 0, \qquad \textup{for all $1 - \delta \le |x| \le 1$}.
\]
This proves \eqref{eq:cone}. 

Let
\[
   d = \frac{1}{2} \min \{v(x) - u(x) : \; |x| \le 1 - \delta\}. 
\]
Then $d > 0$, because of \eqref{eq:case2}. On the other hand, by the uniform continuity, 
\[
   \sup_{\bar{B}_1}|v - v_q| \le C (2 - q), \qquad \textup{for all $(2 +\sqrt{2})/2 < q < 2$}.
\]
where $C>0$ is a constant independent of $q$ and $\delta$. Now let 
\[
			\epsilon = \min\{d/C, (2 -\sqrt{2})/2\} > 0.
\]
We then have for any $2 - \epsilon < q < 2$,
\[
   v_q \ge v - C \epsilon > v - d > v - (v -u) = u, \qquad \textup{for all $|x| \le 1 - \delta$}.
\]
Moreover, combining \eqref{eq:vcone} with \eqref{eq:cone} gives that
\[
   v_q (x) \ge |x| > u(x), \qquad \textup{for all $1 - \delta \le |x| < 1$}.
\]
Therefore, we conclude that for each $2 - \epsilon < q < 2$,
\[
   v_q(x) > u(x), \qquad \textup{for all $|x| \le 1$}.
\]
This proves the claim, and hence, the proof is completed.
\qed
\end{pfthmhypmeanPDE}
\begin{remark}
  As pointed out by Rick Schoen, if the function $u$ in Theorem~\ref{th:hypmeanPDE} satisfies in addition that $\partial u /\partial r = \partial v / \partial r$ on the boundary (which holds if $\partial M$ is totally geodesic in $M$), then \eqref{eq:case2} can be immediately ruled out by Theorem~\ref{th:smax2} \eqref{th:smax2:2}.
\end{remark}

Theorem~\ref{th:hyp2graph} then follows immediately from Proposition~\ref{pr:hypRH} and \ref{pr:hypmodel}, and Theorem~\ref{th:hypmeanPDE}. There is also a $k$-convex version of the rigidity theorem, which also follows from Proposition~\ref{pr:hypRH} and \ref{pr:hypmodel}, Theorem~\ref{th:hypmeanPDE}, together with  Maclaurin's inequality \eqref{eq:Mac}. 
\begin{theorem} \label{th:hypkgraph}
   Let $u \in C^2(B_1) \cap C^1(\bar{B}_1)$ satisfy that $u > 0$ in $B_1$ and $u = 1$ on $\partial B_1$. Let $M_u$ be the graph of $u$ over $\bar{B}_1$. Assume for some $1 \le k \le n$ that 
   \[
      \sigma_k(\kappa) \ge 2^{k/2}\binom{n}{k},
   \]
   and that $M_u$ is $k$-convex if $k \ge 3$, then $M_u$ is isometric to $\mathbb{S}^n_+$.
\end{theorem}

Next, we generalize the rigidity results to a hypersurface in $\mathbb{H}^{n+1}$. Similar to the Euclidean case, we will introduce the hyperbolic incorporation condition. Let $\mathfrak{C}_+$ be the upper solid hemicone, i.e.,
\[
		\mathfrak{C}_+ = \{ ( x^1, \dots, x^{n+1} ) \in \mathbb{H}^{n+1} : x^{n+1} \ge r \mbox{ and } x^{n+1} > 1\}.
\]
\begin{defn} \label{de:hypincp}
  Let $M \subset \mathbb{H}^{n+1}$ be a compact $C^2$ hypersurface with boundary $\partial M$. $M$ is said to have the \emph{hyperbolic incorporation condition}, 
  if $M$ satisfies the following three conditions:
  \begin{enumeratei}
     \item \label{de:hyp:1} The boundary $\partial M$ is diffeomorphic to $\mathbb{S}^{n-1}$.
     \item \label{de:hyp:2} $\partial M \subset \mathbb{R}^n \times \{1\}$, and $B_1$ is contained in the region enclosed by $\partial M$ in $\mathbb{R}^{n} \times \{1 \}$.
     \item \label{de:hyp:3} $M \cap \mathfrak{C}_+ = \emptyset$.
  \end{enumeratei}
\end{defn}
The rigidity theorem for hypersurfaces is as follows:
\begin{theorem} \label{th:hypmean}
   Let $M \subset \mathbb{H}^{n+1}$ be a compact $C^2$ hypersurface with boundary $\partial M$. Assume that $M$ satisfies the hyperbolic incorporation condition. If $M$ has hyperbolic mean curvature $H \ge \sqrt{2}n$, then $M$ is isometric to $\mathbb{S}^n_+$.
\end{theorem}
\begin{pfthmhypmean}
  The proof is similar to that of Theorem~\ref{th:hypmeanPDE}, so we only point out the difference here. 
  We use the following family of functions,
  \[
     v_q (r) = q - \sqrt{\frac{q^2}{2} - r^2}, \qquad \textup{for all $r \le q/\sqrt{2}$},
  \]
  where $q > 0$ serves as a parameter. Denote by $S(q)$ the graph of $v_q$. Again we start with a very large $q$ so that $S(q)$ has no intersection with $M$. Then we move $S(q)$ downward, by continuously decreasing $q$, until $q = q_0$ when $S(q)$ begins to intersect $M$. 
  
  Observe that $S(q_0) \cap M$ is contained in the interior of $S(q_0)$. To see this, notice the fact that
  \[
     \partial S(q_0) = \{(x^1,\ldots,x^{n+1}): r = q_0/\sqrt{2}, x^{n+1} = q_0\} \subset S(3 q_0).
  \]
   Thus, if there exists an $x \in \partial S(q_0) \cap M$, then $S(3q_0)$ should also intersect $M$ at $x$. This contradicts our choice of $q_0$.
   
   For any $y \in S(q_0)\cap M$, the observation enables us to locally write $M$ near $y$ as a graph over a small ball $V\subset B_1$. The \eqref{de:hyp:3} of incorporation condition assures that the local graph has the right sign for its mean curvature. Thus, similar to the proof of Theorem~\ref{th:hypmeanPDE}, we can apply Theorem~\ref{th:smax2} to obtain that $q_0 = 2$, and that either $M$ is isometric to $\mathbb{S}^n_+$, or $S(2)$ only intersects $M$ at $\partial M$. It remains to rule out the latter case. The process goes the same as that of Theorem~\ref{th:hypmeanPDE}, except that the Case $2^{\circ}$ here is excluded by virtue of the above observation. The difference is that a hypersurface can be vertical at the interior of $B_1$, in contrast to a graph.
\qed   
\end{pfthmhypmean}

Finally, Theorem~\ref{th:hypkcvx} follows as a corollary of Theorem~\ref{th:hypmean} and Maclaurin's inequality.

\section{Appendix: Strong maximum principles.}

Throughout the appendix, we denote by $B_r$ the open ball in $\mathbb{R}^n \times \{ 0 \}$ of radius $r$ centered at the origin, and denote $B = B_1$ for simplicity. A function with subscripts stands for the derivatives of the function. For example,
\[
   u_i = \frac{\partial u}{\partial x^i}, \quad u_{ij} = \frac{\partial^2 u}{\partial x^i x^j}.
\]
Consider
\[
   Q(u) = \sum_{i,j=1}^n \tilde{a}^{ij}(u, Du) u_{ij} + \tilde{b}(Du) \qquad \textup{in $B$},
\]
for all $u \in C^2(B)\cap C^0(\bar{B})$.
Here
\[
   \tilde{a}^{ij}(t, p) = \frac{z(t)}{\sqrt{1 + |p|^2}}\left( \delta_{ij} - \frac{p^i p^j}{1 + |p|^2}\right)
\]
and
\[
   \tilde{b}(p) = \frac{b_0}{\sqrt{1 + |p|^2}}, \qquad \textup{for all $p = (p^1,\ldots,p^n) \in \mathbb{R}^n$},
\]
in which $b_0$ is a constant, and $z = z(t)$ is a smooth function defined on a domain in $\mathbb{R}$.
In particular, $Q$ is the Euclidean mean curvature operator $H_0$ if $b_0 =0$ and $z \equiv 1$; $Q$ is the hyperbolic mean curvature operator $H$ if $b_0 = n$ and $z(t) = t$. 
\begin{theorem}\label{th:smax2}
  Let $\varphi, \psi \in C^2(B)\cap C^0(\bar{B})$ with $z(\psi)>0$ on $\bar{B}$.
  Let $V$ be an open ball in $B$ ($V$ could be $B$ itself) 
  such that
  \[
      Q(\varphi) \le Q(\psi) , \quad \textup{and \quad $\varphi \ge \psi$}, \qquad \textup{in $V$}.
  \]
  \begin{enumerate}
  \item \label{th:smax2:1} If $\varphi = \psi$ at some interior point of $V$, then
  \[
     \varphi \equiv \psi \qquad \textup{on $\bar{V}$}.
  \]
  \item \label{th:smax2:2} If $\varphi \in C^2(\bar{V})$, $\psi \in C^2(V)\cap C^1(\bar{V})$, $\varphi > \psi$ in $V$, and $\varphi = \psi$ at some $x_0 \in \partial V$, then
  \[
     D (\psi - \varphi)(x_0) \cdot \eta > 0,
  \]
  where $\eta$ is the outward unit normal vector at $x_0$ to $\partial V$.
  \end{enumerate}
\end{theorem}
\begin{remark}
  Part~\eqref{th:smax2:2} of the theorem is the boundary point lemma for quasilinear operators. The regularity of $\varphi$ and $\psi$ can be replaced by that at least one of $\varphi$ and $\psi$ belongs to $C^{2}(V)\cap C^{1,1}(\bar{V})$, while the other is in $C^2(V)\cap C^{0,1}(\bar{V})$ and at which $z > 0$.
\end{remark}
The proof of Theorem~\ref{th:smax2} is based on Hopf's strong maximum principle and boundary point lemma. 
Let $V$ be the open ball in $B$, and $L$ be the linear operator given by 
\begin{equation} \label{eq:defL2}
   L h = \sum_{i,j=1}^n a^{ij} h_{ij} + \sum_{i=1}^n b^i h_i + c h \qquad \textup{in $V$},
\end{equation}
for all $h \in C^2(V)$. Assume that the coefficient matrix $(a^{ij})$ is everywhere positive definite in $V$, and that $a^{ij}$, $b^i$, and $c$ are continuous in $V$ for all $1 \le i, j \le n$. 
A special case of Hopf's strong maximum principle can be stated as follows:
\begin{lemma} \label{le:Hopf2}
  Let $L$ be the operator given by \eqref{eq:defL2}, and $h \in C^2(V)$ such that $h \le 0$ and $L h \ge 0$ in $V$. If $h = 0$ at some interior point of $V$, then $h \equiv 0$ on $V$.
\end{lemma}
Lemma~\ref{le:Hopf2} follows, in turn, from the following boundary point lemma due to Hopf.
\begin{lemma}\label{le:Hopfbdy}
 Let $L$ be the operator given by \eqref{eq:defL2}, and $U$ be an open ball in $V$ such that $a^{ij}$, $b^i$, $c \in C^0(\bar{U})$ for all $1 \le i,j\le n$, and that there exists a positive constant $\theta$ such that
\begin{equation} \label{eq:theta2}
    \sum_{i,j=1}^n a^{ij} \xi_i\xi_j \ge \theta |\xi|^2, 
\end{equation}
for all $x \in U$ and $\xi = (\xi_1,\ldots,\xi_n) \in \mathbb{R}^n$. Let $h \in C^2(U) \cap C^1(\bar{U})$ such that $L h \ge 0$ in $U$. Suppose that for some $x_0 \in \partial U$,
 \[ 
    h(x_0)= 0 > h(x), \qquad \textup{for all $x \in U$}.
 \]
  Then,
 \[
     Dh(x_0) \cdot \mu > 0,
 \]
 where $\mu$ is the outward unit normal vector to $U$ at $x_0$.
 \end{lemma}
\begin{pfleHopfbdy}
   We assume without loss of generality that $U = B_{\delta}$ for some $0 < \delta < 1$. Define
   \[
      w(x) = e^{-\lambda |x|^2} - e^{- \lambda \delta^2}, \qquad \textup{for all $x \in B_{\delta}$},
   \]
   where $\lambda > 0$ is a constant yet to be determined. 
   Notice that
   \begin{align*}
     & (L - |c|) w \\
     & \ge e^{-\lambda |x|^2} \left[4 \lambda^2 \sum_{i,j=1}^n a^{ij}x_i x_j - 2 \lambda \left(\sum_{i=1}^n b^i x_i + \sum_{i=1}^n a^{ii} \right) - (|c| - c)\right] \\
     & \ge e^{-\lambda |x|^2} \left[ 4 \lambda^2 \theta |x|^2 - 2 \lambda C (|x| + 1) - C \right].
   \end{align*}
   Here $\theta > 0$ is given by \eqref{eq:theta2}, and $C > 0$ is a constant depending on the $C^0(\bar{V})$--norms of $a^{ii}$, $b^i$, and $c$.
   Now consider the annulus $A = B_{\delta}\setminus B_{\delta/2}$. We can choose a sufficiently large constant $\lambda = \lambda(\theta, \delta, C)$ such that
   \begin{equation} \label{eq:subw}
      (L - |c|) w  > 0 \qquad \textup{on $\bar{A}$}.
   \end{equation}
   Since $h(x_0) = 0 > h(x)$ on $\partial B_{\delta/2}$, there is a constant $\varepsilon > 0$ such that
   \begin{equation} \label{eq:bdyuw2}
      h(x_0) = 0 \ge h(x) + \varepsilon w(x), 
   \end{equation}
   for all $x \in \partial B_{\delta/2}$. Note that \eqref{eq:bdyuw2} also holds on $\partial B_{\delta}$, where $w$ is identically zero. 
   On the other hand, we have by \eqref{eq:subw} that
   \[
      (L - |c|) (h + \varepsilon w) > - |c| h \ge 0 \qquad \textup{on $\bar{A}$}.
   \]
   It follows from the usual maximum principle that
   \[
       h + \varepsilon w \le 0 \qquad \textup{on $\bar{A}$}.
   \]
   But $h (x_0) + \varepsilon w (x_0) = 0$. Taking the normal derivative at $x_0$ yields that
   \[
       \frac{\partial u}{\partial \mu}(x_0) + \varepsilon \frac{\partial w}{\partial \mu}(x_0) \ge 0.
   \]
   Thus,
   \[
      \frac{\partial u}{\partial \mu}(x_0) \ge - \varepsilon \frac{\partial w}{\partial \mu}(x_0) = 2 \lambda \varepsilon \delta e^{-\lambda \delta^2} > 0.
   \]
   \qed
\end{pfleHopfbdy} 

\begin{pfleHopf2}
  Let
  \[
     E = \{ x \in V \mid h(x) = 0\}.
  \]
  Then $E$ is relatively closed in $V$. By the assumption $E$ is nonempty. We need to show that $E = V$. Suppose not. We can then choose a point $y \in V \setminus E$ such that
  \[
      d(y, E) < d (y, \partial V)/2.
  \]
  Consider the largest open ball $U \subset V \setminus E$ centered at $y$. Then, by the construction $\bar{U} \subset V$ and $\partial U$ must intersect $E$ at some point $x_0$. Thus, this implies that $a^{ij}$, $b^i$, $c \in C^0(\bar{U})$ for all $i,j = 1,\ldots,n$, that \eqref{eq:theta2} holds for some constant $\theta>0$ depending on $U$, and that $h \in C^2(\bar{U})$ and
  \[
      h(x_0) = 0 > h(x), \qquad \textup{for all $x \in U$}.
  \]
  Applying Lemma~\ref{le:Hopfbdy} yields that
  \[
     D h (x_0) \neq 0.
  \]
  This is a contradiction, since $x_0$ is an interior maximum point of $h$ in $V$.
  \qed
\end{pfleHopf2}

\begin{pfthmsmax2}
  Note that
  \begin{align*}
     & Q (\psi) - Q (\varphi) \\
     & = \sum_{i,j=1}^n \tilde{a}^{ij}(\psi, D\psi)(\psi_{ij} -  \varphi_{ij}) + \sum_{i,j=1}^n\varphi_{ij} \big[\tilde{a}^{ij}(\psi, D\psi) - \tilde{a}^{ij}(\psi,D\varphi) \big] \\
     & \quad + \tilde{b}(D\psi) - \tilde{b}(D\varphi) + \sum_{i,j=1}^n \varphi_{ij} \big[\tilde{a}^{ij}(\psi, D\varphi) - \tilde{a}^{ij}(\varphi,D\varphi) \big].
  \end{align*}
  Let $h = \psi - \varphi$. We can rewrite 
  \begin{equation*}
     0 \le Q(\psi) - Q(\varphi) = \sum_{i,j = 1}^n a^{ij}h_{ij} + \sum_{i=1}^n b^i h_i + c h = Lh,
  \end{equation*}
  where
  \begin{align*}
     a^{ij} & = \tilde{a}^{ij}(\psi, D\psi), \\
     b^i    & = \sum_{l,m=1}^n \varphi_{lm} \int_0^1 \frac{\partial \tilde{a}^{lm}}{\partial p^i}\big(\psi, tD\psi + (1 - t) D\varphi\big) dt \\
     				& \quad + \int_0^1 \frac{\partial \tilde{b}}{\partial p^i}\big(tD\psi + (1-t)D\varphi \big)dt, \\
     c			& = \sum_{l,m=1}^n \varphi_{lm} \int_0^1 \frac{\partial \tilde{a}^{lm}}{\partial z} \big(t\psi + (1 - t) \varphi, D\varphi \big) dt,
  \end{align*}
  for all $i, j = 1, \ldots, n$. Then since $\varphi, \psi \in C^2(B)$, $a^{ij}$, $b^i$, $c$ are continuous in $B$ for all $i,j=1,\ldots, n$. Furthermore, we have
  \begin{equation*} 
     \sum_{i,j=1}^n a^{ij} \xi_i \xi_j \ge \frac{z(\psi)}{(1 + |D\psi|^2)^{3/2}} |\xi|^2 > 0
  \end{equation*}
  for all $x \in B$ and $\xi \in \mathbb{R}^n \setminus \{0\}$. Thus, the first part of Theorem~\ref{th:smax2} follows immediately from Lemma~\ref{le:Hopf2}. For the second part, since $\varphi \in C^2(\bar{V})$ and $\psi \in C^1(\bar{V})$, we have $a^{ij}, b^i, c \in C^0(\bar{V})$ for all $1 \le i,j\le n$. Furthermore, \eqref{eq:theta2} holds as we can take
  \[
     \theta = \min_{\bar{V}} \frac{z(\psi)}{(1 + |D\psi|^2)^{3/2}} > 0.
  \]
  Now applying Lemma~\ref{le:Hopfbdy} with $U = V$ yields the result.
  \qed
\end{pfthmsmax2}

\textbf{Added in Proof (May 21, 2010):} 
The assumption that \emph{$M$ being $k$-convex when $k \ge 3$} can in fact be removed from Theorem 1.1, Theorem 1.3, and Theorem 2.1. We are very grateful to Professor Pengfei Guan for pointing this out to us. 

The reason is as follows: First, observe that $M$ is $k$-convex everywhere on $M$ as long as it is $k$-convex at \emph{one point}. This is due to the well--known fact (see, for example, \cite{CNS} and \cite[p. 51]{HuiSin}): The convex cone $\Gamma_k \equiv \{ \kappa \in \mathbb{R}^n : \sigma_1(\kappa) >0, \sigma_2(\kappa) >0, \dots, \sigma_k(\kappa) >0 \}$ is the \emph{connected component} of $\{\kappa \in \mathbb{R}^n : \sigma_k (\kappa) > 0\}$ containing $\{\kappa \in \mathbb{R}^n : \kappa_1>0, \dots, \kappa_n >0 \}$. 
Thus, it suffices to find a point $p$ in $M$ so that $\kappa(p) \in \Gamma_k$. By the incorporation condition, $\partial M$ lies in a hyperplane, denoted by $L$. Because $M$ is bounded, we can start from 
a hyperplane which is parallel to $L$ and disjoint from $M$, and translate the hyperplane upward  until it begins to contact $M$ at some interior point $p_0$.
Then, the principal curvature $\kappa_i \ge 0$ at $p_0$ for all $1 \le i \le n$ (in particular, in hyperbolic space we have $\kappa_i > 0$ at $p_0$ for all $i$). By the assumption, $\sigma_k(\kappa) > 0$ on $M$. Applying Maclaurin's inequality yields that
\[
   \sigma_j^{1/j}(\kappa) \ge \sigma_k^{1/k}(\kappa)  > 0 \quad \textup{at $p_0$, \quad for all $1 \le j <k$.}
\]
This shows that $\kappa (p_0) \in \Gamma_k$.

\end{document}